\documentclass[12pt]{article}

\usepackage[margin=1in]{geometry}  
\usepackage{graphicx}              
\usepackage{amsmath}               
\usepackage{amsfonts}              
\usepackage{amsthm}                
\usepackage{ amssymb}      
\usepackage{hyperref}
\newtheorem{theorem}{Theorem}[section]
\newtheorem{lemma}[theorem]{Lemma}

\newtheorem{corollary}[theorem]{Corollary}
\theoremstyle{definition}
\newtheorem{definition}[theorem]{Definition}

\theoremstyle{remark}

\numberwithin{equation}{section}

\begin{document}
\title{\bf Some Fundamental Theorems of Functional Analysis with Bicomplex and Hyperbolic Scalars}
\author{\textbf{Heera Saini Aditi Sharma and Romesh Kumar}}
\date{\textbf{}}
\vspace{0in}
\maketitle
$\textbf{Abstract.}$ We  discuss some properties of  linear functionals on   topological hyperbolic   and  topological bicomplex   modules. The hyperbolic and bicomplex analogues of the uniform boundedness principle, the open mapping theorem, the closed graph theorem and the  Hahn Banach separation theorem  are  proved.\\\\
 $\textbf{Keywords.}$  Bicomplex modules, hyperbolic modules,  topological bicomplex modules, topological hyperbolic modules,   hyperbolic convexity,  hyperbolic-valued Minkowski functionals, continuous  linear functionals, Baire category theorem,  uniform boundedness principle, open mapping theorem, inverse mapping theorem, closed graph theorem, Hahn Banach separation theorem.

\begin{section} {Introduction and Preliminaries}
 \renewcommand{\thefootnote}{\fnsymbol{footnote}}
In this paper, we prove several core results of bicomplex functional analysis. The four basic principles, the Hahn Banach separation theorem for topological hyperbolic and topological bicomplex modules, the open mapping theorem, the closed graph theorem and the uniform boundedness principle for   $F$-bicomplex and   $F$-hyperbolic modules with hyperbolic-valued norm have been presented. Hahn Banach theorem is one of the most fundamental results in functional analysis. The theorem for real vector spaces was first proved by H. Hahn \cite{HH}, rediscovered in its present, more general form by S. Banach \cite{BA}. Further, it was generalized to complex vector spaces by H. F. Bohnenblust and A. Sobczyk \cite{BS}; and Soukhomlinoff \cite{souk} who also considered the case of vector spaces with quaternionic scalars. The Hahn-Banach theorem for bicomplex functional analysis with real-valued norm was proved in \cite{LL}. Recently, Hahn Banach theorem for bicomplex functional analysis with hyperbolic-valued norm was proved in \cite{Hahn}, which is in analytic form, involving the existence of extensions of a linear functionals. The uniform boundedness principle was proved by Banach and Steinhaus \cite{BAHS}.  Some related results to this principle were already proved by Lebesgue \cite{HL}, Hahn \cite{HH0}, Steinhaus \cite{stein}, and Saks and Tamarkin \cite{sktm}. The open mapping theorem was first proved in 1929 by Banach \cite{BA}. Later, in 1930, a different proof was given by Schauder \cite{schau}.  There is a vast literature on generalizations of Hahn Banach theorem, the open mapping theorem and the uniform boundedness principle in different directions.  

The main idea of the paper is to extend the work of  \cite{Hahn}  to the more general case of topological hyperbolic  and  topological bicomplex  modules. Both \cite{LL} and \cite{Hahn} have presented the Hahn Banach theorem in analtyic form with real-valued and hyperbolic-valued norm respectively. We are interested to present this result in its geometric form which has been done in Section 6, 7 and 8. 

The work in this paper is organized in the following manner.  In Section 2, we give  a brief discussion on  topological hyperbolic  modules.  Section 3 discusses some properties of  linear functionals on  topological hyperbolic   and  topological bicomplex  modules. In Section 4, we first prove the  Baire category theorem for  $\mathbb{D}$-metric spaces, which is then used to prove the uniform boundedness principle  for   $F$-bicomplex and   $F$-hyperbolic modules. The open mapping theorem, the inverse mapping theorem and the closed graph theorem   for  $F$-hyperbolic and  $F$-bicomplex modules have been proved in Section 5. In Sections 6 and 7, we  present a geometric form of Hahn Banach theorem  for  topological hyperbolic   and  topological bicomplex  modules and we call them hyperbolic Hahn Banach separation theorem and bicomplex  Hahn Banach separation theorem respectively. Section 8 deals with another geometric form of Hahn Banach theorem for hyperbolic  modules in terms of hyperbolic hyperplanes.

Now, we summarize some basic properties of bicomplex numbers. The set  $\mathbb{BC}$ of bicomplex numbers  is defined as 
$$
\mathbb{BC} =\left\{Z=w_{1}+jw_{2}\;|\;w_1 ,w_2 \in \mathbb{C}(i)\right\},
$$
where $i$ and $j$ are imaginary units such that $ ij=ji,  i^2=j^2=-1$ and $\mathbb{C}(i)$ is the set of complex numbers with the imaginary unit $i$. The set $\mathbb{BC}$ of bicomplex numbers  forms a  ring under the usual addition and multiplication of bicomplex numbers. 
Moreover,  $\mathbb{BC}$ is a module over itself.
The product of imaginary units $i$ and $j$ defines a hyperbolic unit $k$ such that $k^2=1$. The product of all units is commutative and satisfies $$ij=k, ~~ik=-j~~ \textmd{and}~~ jk=-i.$$
The set $\mathbb{D}$  of hyperbolic numbers  is defined as
$$\mathbb{D}=\left\{ \alpha = \beta_1 +k \beta_2 : \beta_1 , \beta_2 \in \mathbb{R} \right\}.$$   The set $\mathbb{D}$  is a ring and a  module over itself.
Since the set $\mathbb{BC}$ has two imaginary units $i$ and $j$, two conjugations can be defined for bicomplex numbers and composing these two conjugations, we obtain a third one. We define these conjugations as follows:
\begin{enumerate}
\item[(i)] $Z^{\dagger_1}=\overline{w_1}+j\overline{w_2},$ 
\item[(ii)] $Z^{\dagger_2}={w_1}-j{w_2},$ 
\item[(iii)] $Z^{\dagger_3}= \overline{w_1}-j\overline{w_2},$
\end{enumerate}
where $\overline{w_1}, \overline{w_2}$ are   respectively the usual complex conjugates of ${w_1}, {w_2} \in \mathbb C(i)$.  For bicomplex numbers we have  the following three moduli:
\begin{enumerate}
\item[(i)] $|Z|^2_i=Z\;.\;Z^{\dagger_2}={w_1}^2+{w_2}^2 \in \mathbb{C}(i),$
\item[(ii)] $|Z|^2_j=Z\;.\;Z^{\dagger_1}= (|{w_1}|^2-|{w_2}|^2) + j(\;2\;\mathrm{Re}({w_1}\;\overline{{w_2}})) \in \mathbb{C}(j),$  
\item[(iii)] $|Z|^2_k=Z\;.\;Z^{\dagger_3}=(|{w_1}|^2+|{w_2}|^2) +k(\;-2\;\mathrm{Im}({w_1}\;\overline{{w_2}})) \in \mathbb{D}.$ 	 
\end{enumerate}
 The hyperbolic numbers $e_1$ and $e_2$  defined as
 $$e_1= \frac{1+k}{2} \;\; \textmd{and}\;\;\ e_2=\frac{1-k}{2}\;,$$
 are  linearly independent in the $\mathbb{C}$(i)- vector space $\mathbb{BC}$ and satisfy the following properties:
 $${e_1}^2=e_1,\;\;\; {e_2}^2=e_2,\;\; {e_1}^{\dagger_{3}}= e_1 , \;\;\;{e_2}^{\dagger_{3}}= e_2 ,\;\;\; e_1+e_2=1\;\;\; e_1\cdot e_2=0\;.$$ Any bicomplex number $Z={w_1}+j{w_2}$ can be uniquely written as 
\begin{equation}\label{eq1000} Z=e_1z_1+e_2z_2 \;,
\end{equation}  where $z_1=w_1 -iw_2 $ and $z_2=w_1 +iw_2 $ are elements of  $\mathbb{C}(i)$. Formula (\ref{eq1000}) is called the idempotent representation of a bicomplex number $Z$.
The sets $e_1\mathbb {BC}$ and $e_2\mathbb {BC}$ are ideals in the ring $\mathbb {BC}$ such that $$e_1\mathbb {BC} \; \cap \;e_2\mathbb {BC}= \left\{0\right\}$$ and 
\begin{equation}\label{eq1001} \mathbb {BC}= e_1\mathbb {BC}+e_2\mathbb {BC}\;.
\end{equation} Formula  (\ref{eq1001}) is called the idempotent decomposition of $\mathbb {BC}$. \\
For bicomplex numbers, $Z=e_1 z_1+e_2 z_2$ and  $Z'=e_1z'_1+e_2z'_2$, the product $Z\cdot Z'=1$ if and only if $e_1 z_1 z'_1+e_2 z_2 z'_2= e_1 +e_2$. We, therefore, have  $Z\cdot Z'=1$  if and only if $ z_1 z'_1 =1$ and $z_2 z'_2=1.$ Thus, $Z$ is invertible if and only if both $z_1$ and $z_2$ are non-zero and the inverse $Z^{-1}$ is equal to $e_1z^{-1}_1+e_2z^{-1}_2$. A nonzero bicomplex number $Z=e_1 z_1+e_2 z_2$  does not have an inverse if either $z_1$ is zero or $z_2$ is zero and such a $Z$ is called a zero divisor.  The set $\mathcal {NC}$ of all zero divisors  of $\mathbb {BC}$ is, thus, given by
 $$\mathcal {NC} =\left\{Z=e_1 z_1+e_2 z_2 \;|\; z_1 =0\; \textmd{or}\; z_2 =0 \right\},$$ and is called the null cone.\\
 The hyperbolic-valued or $\mathbb D$-valued norm $|Z|_k$ of a bicomplex number $Z=e_1z_1+e_2z_2$ is defined as  $$|Z|_k=e_1|z_1|+e_2|z_2|.$$ \\
A hyperbolic number $\alpha = \beta_1 +k \beta_2 $ in idempotent representation  can be  written as 
$$ \alpha = e_1 \alpha_1 + e_2 \alpha_2, \;\;\; $$
where $\alpha_1 = \beta_1 + \beta_2 $ and $\alpha_2 = \beta_1 - \beta_2 $ are real numbers. We say that $\alpha$ is a positive hyperbolic number if $\alpha_1  \geq 0$ and $ \alpha_2 \geq 0 .$ Thus, the set of positive hyperbolic numbers $\mathbb{D}^{+}$ is given by
$$\mathbb{D}^{+} = \{ \alpha = e_1 \alpha_1 + e_2 \alpha_2 : \;\; \alpha_1 \geq 0,  \alpha_2 \geq 0\}.$$
For $\alpha, \gamma \in \mathbb{D},$ define a relation $\leq^{'}$ on $\mathbb{D}$ by $\alpha \; \leq^{'}\; \gamma$ whenever $\gamma - \alpha \in \mathbb{D}^{+}$. This relation is reflexive, anti-symmetric as well as transitive and hence defines a partial order on $\mathbb{D}.$   If we write the hyperbolic numbers $\alpha$ and $\gamma$ in idempotent representation as  $\alpha = e_1 \alpha_1 + e_2 \alpha_2 $ and  $\gamma = e_1 \gamma_1 + e_2 \gamma_2$, then $\alpha \; \leq^{'}\; \gamma$ implies that  $ \alpha_1  \leq \gamma_1$ and $ \alpha_2 \leq \gamma_2. $  By  $\alpha \; <' \; \gamma$, we mean  $ \alpha_1  < \gamma_1$ and $ \alpha_2 < \gamma_2. $ For further details on partial ordering on $\mathbb{D}$ one can refer \cite [Section 1.5] {YY}. \\
If $A \subset \mathbb{D}$ is $ \mathbb{D}$-bounded from  above, then the $\mathbb{D}$-supremum of $A$ is defined as the smallest member of the set of all upper bounds of $A$. In other words,  a hyperbolic number 
$\lambda$ is an upper bound of the set $A$, can be described by the following two properties:
\begin{enumerate}
\item[(i)] $\alpha \leq' \lambda$, for each $\alpha \in A$.
\item[(ii)]  For any $\epsilon >'0$, there exists $\beta \in A$ such that $\beta >' \lambda - \epsilon$.
\end{enumerate}
That is, the hyperbolic number $\lambda = e_1 \lambda_1 + e_2 \lambda_2 $, where $\lambda_1$ and $\lambda_2$ are real numbers, is the $\mathbb{D}$-supremum of $A$ if 
\begin{enumerate}
\item[(i)] $e_1 \alpha_1+ e_2 \alpha_2 \leq'e_1 \lambda_1 + e_2 \lambda_2$, for each $\alpha = e_1 \alpha_1+ e_2 \alpha_2 \in A$.
\item[(ii)]  For any $\epsilon= e_1 \epsilon_1 + e_2 \epsilon_2  >' 0$, there exists $\beta = e_1 \beta_1 + e_2 \beta_2  \in A$ such that $e_1 \beta_1 + e_2 \beta_2 >' e_1 \lambda_1 -  e_1 \epsilon_1 + e_2 \lambda_2 - e_2 \epsilon_2 $.
\end{enumerate} 
Let $A_1 = \{\lambda_1 \; : \; e_1 \lambda_1  + e_2 \lambda_2  \in A\}$  and  $A_2 = \{\lambda_2 \; : \; e_1 \lambda_1  + e_2 \lambda_2  \in A\}$ . 
Then, $ \alpha_1 \leq  \lambda_1 $ for each $\alpha_1 \in A_1$,   $ \alpha_2 \leq \lambda_2$ for each $\alpha_2 \in A_2$   and for $\epsilon_1>0$, $\epsilon_2 >0$, there exists $ \beta_1 \in A_1$ and $ \beta_2  \in A_2$ such that $\beta_1  > \lambda_1 - \epsilon_1 $ and $\beta_2  > \lambda_2 - \epsilon_2 $, showing that $\lambda_1$ is the supremum of $A_1 $  and $\lambda_2$ is the supremum of $A_2 $ and hence we obtain
 $$ \sup_{\mathbb{D}}A = \sup A_1 e_1 +  \sup A_2 e_2,$$  Similarly, $\mathbb{D}$-infimum of  a  $ \mathbb{D}$-bounded below set $A$ is defined as
$$ \inf_{\mathbb{D}} A = \inf A_1 e_1 +  \inf A_2 e_2,$$ where $A_1$ and $A_2$ are as defined above.
\\
A $\mathbb{B}\mathbb{C}$-module (or $\mathbb{D}$-module) $X$ can be written as
\begin{equation}\label{eq1002} X=e_1X_1+e_2X_2\;,
\end{equation} where $X_1=e_1X$ and $X_2=e_2X$ are $\mathbb{C}(i)$-vector  (or $\mathbb{R}$-vector) spaces. Formula (\ref{eq1002}) is called the idempotent decomposition of $X$. Thus, any $x$ in $X$ can be uniquely written as $x=e_1x_1+e_2x_2$ with $x_1\in X_1,\;x_2\in X_2$.
\end{section}

 \begin{section} {Topological  Hyperbolic Modules} 
Topological bicomplex modules have been defined and discussed thoroughly in \cite{HS}.  In  this section, we introduce the concepts of   topological  hyperbolic modules and absorbedness  in hyperbolic modules. Hyperbolic convexity and hyperbolic-valued Minkowski functionals  in  hyperbolic modules have already  been defined in \cite{Hahn}.  The proofs of all the theorems in this section  are  on similar lines as  of their bicomplex counterparts, (see,  \cite [Section 2] {HS}), so we omit them.
\begin{definition}\label{d1}
Let $X$ be a $\mathbb{D}$-module and $\tau$ be a Hausdorff topology on $X$ such that the operations
\begin{enumerate}
\item[(i)] $+\; : X \times X \longrightarrow X$  and
\item[(ii)] $ \cdot\; : \mathbb{D} \times X \longrightarrow X$
\end{enumerate}
are continuous. Then the pair $(X, \tau )$ is called a topological hyperbolic module or topological $\mathbb{D}$-module.
\end{definition}  Let $(X, \tau )$ be a topological $\mathbb{D}$-module. Write $$X= e_1 X_1 + e_2 X_2, $$ where  $X_1=e_1X$ and $X_2=e_2X$  are $\mathbb{R}$-vector spaces as well as  $\mathbb{D}$-modules . Consider $X_l$ to be  $\mathbb{R}$-vector space, for $l=1,2$. Then 
  $\tau_l = \{ e_l G \; : \; G \in \tau \}$  
is a Hausdorff topology on $X_l$  such that the operations
\begin{enumerate}
\item[(i)] $+\; : X_l \times X_l \longrightarrow X_l$  and
\item[(ii)] $ \cdot\; : \mathbb{R} \times X_l \longrightarrow X_l$
\end{enumerate}
are continuous for $l= 1, 2$. Therefore, $(X_l, \tau_l)$ is a topological $\mathbb{R}$-vector space for $l=1, 2.$ Similarly,  $(X_l, \tau_l)$ becomes a topological $\mathbb{D}$-module, when $X_l$ is considered to be  $\mathbb{D}$-module, for $l=1, 2.$ 

\begin{definition}\label{d2}  \cite [Definition 17]{Hahn} 
Let $B$ be a subset of a $\mathbb{D}$-module $X$. Then $B$ is called a $\mathbb{D}$-convex set if  $x, y \in X$ and $\lambda \in \mathbb{D}^{+}$ satisfying $0 \leq ' \lambda \leq ' 1$ implies that $\lambda x + (1- \lambda) y \in B.$
\end{definition}
\begin{theorem} \label{Hb1}
Let $B$ be a  $\mathbb{D}$-convex subset of $\mathbb{D}$-module  $X$. Then $B$ can be written as  $$B \; =\; e_1 B + e_2 B.$$
\end{theorem}
\begin{definition}\label{d3}
Let $B$ be a subset of a $\mathbb{D}$-module $X$. Then $B$ is called a $\mathbb{D}$-absorbing set if for each $x \in X$, there exists $ \epsilon > ' 0$ such that $\lambda x \in B$ whenever  $ 0 \; \leq ' \;  \lambda \;  \leq ' \;  \epsilon$.
\end{definition}
\begin{theorem} \label{Hb2}
Let  $(X, \tau ) $ be a topological  $\mathbb{D}$-module. Then each neighbourhood of $0$ in $X$ is $\mathbb{D}$-absorbing.
\end{theorem} 
\begin{definition}\label{d4}
Let $B$ be a $\mathbb{D}$-convex, $\mathbb{D}$-absorbing subset of a $\mathbb{D}$-module $X$. Then, the  mapping $q_B \; : \;  X \longrightarrow \mathbb{D}^{+}$ defined by $$ q_B (x) \; = \; \inf_{\mathbb{D}} \{ \alpha >' 0 \; : \; x \in \alpha B  \}, \;\;\; \textmd{for each} \; x \in X,$$ is called hyperbolic-valued  gauge or  hyperbolic-valued Minkowski functional of $B$. 
\end{definition}
For each $x \in X$, we can write $q_B(x) = q_{e_1 B}( e_1 x) e_1 +  q_{e_2 B}( e_2 x) e_2$, where $q_{e_1 B}$ and    $q_{e_2 B}$ are real valued Minkowski functionals on $ e_1 X$ and  $ e_2 X$ respectively. For further details on  hyperbolic-valued Minkowski functionals one can see  \cite{HS}  and \cite{Hahn}.

\end{section} 


\begin{section} {Bicomplex and Hyperbolic Linear Functionals } 
In this section we give some properties of linear functionals on topological $\mathbb{BC}$   and topological   $\mathbb{D}$-modules.

 Suppose $X$ is a $\mathbb{BC}$-module. Then $X$ is also a module over the ring of hyperbolic numbers, i.e.,  $X$ is also a $\mathbb{D}$-module. An additive functional  $f$ on $X$ is said to be $\mathbb{D}$-linear if $f(\alpha x) = \alpha f(x)$ for each $x \in X$ and each $\alpha \in \mathbb{D}$; and $\mathbb{BC}$-linear if $f(\alpha x) = \alpha f(x)$ for each $x \in X$ and each $\alpha \in \mathbb{BC}.$

Let $X$ be a topological  $\mathbb{BC}$-module (or $\mathbb{D}$-module) and $f$ be a  $\mathbb{BC}$-functional (or $\mathbb{D}$-functional) on $X$. Then $f$ is said to be continuous at $x \in X$ if for each $\epsilon >' 0$, there exists a neighbourhood $V \subset X$ of $x$ such that $|f(y)- f(x)|_k < ' \epsilon$ for each $y \in V$. $f$ is said to be continuous on $X$ if $f$ is  continuous at each $x \in X$.\\

 Let $X$ be a $\mathbb{BC}$-module (or $\mathbb{D}$-module) and  $f$ be a $\mathbb{BC}$-linear (or $\mathbb{D}$-linear)  functional on $X$. Then there exist $\mathbb{C}(i)$-linear (or   $\mathbb{R}$-linear)   functionals $f_1$ and $f_2$  on $X$ such that $$ f(x)=e_1 f_1 ( x) + e_2 f_2 ( x), \;\;\; \textmd{for each} \;  x \in X.$$  Since $f$ is $\mathbb{BC}$-linear (or $\mathbb{D}$-linear), for each $x \in X$, we have
\begin{eqnarray*}
f(x) &=& f(e_1 x+e_2 x)\\
&=& e_1 f(e_1 x) + e_2 f (e_2 x)\\
&=& e_1 f_1 (e_1 x) + e_2 f_2 (e_2 x).
\end{eqnarray*}
That is, \begin{equation}\label{Eqhb1}
f(x) = e_1 f_1 (e_1 x) + e_2 f_2 (e_2 x),
\end{equation}
for each $x \in X$. This leads to the following theorem:

\begin{theorem}\label{Thhb1} Let $X$ be a topological $\mathbb{BC}$-module (or $\mathbb{D}$-module) and $f$ be a $\mathbb{BC}$-linear (or $\mathbb{D}$-linear) functional on $X$. Then the following statements are equivalent:
\begin{enumerate}
\item[(i)] $f$ is  continuous  if and only if ${f_1}_{|X_1}$  and ${f_2}_{|X_2}$ are  continuous.
\item[(ii)] $f$ is  $\mathbb{D}$-bounded in some neighbourhood of $0$   if and only if ${f_1}_{|X_1}$  and ${f_2}_{|X_2}$ are bounded in some neighbourhoods of $0$.
\end{enumerate}
\end{theorem}
\begin{proof} $(i)$ Suppose $f$ is continuous. Let $x_1 \in  X_1 = e_1 X, \;  x_2 \in X_2 = e_2 X$ and $\epsilon_1, \epsilon_2 > 0$ be given. Then there exists $x \in X$ such that $x_1 = e_1 x$ and $x_2 = e_2 x$. Set $\epsilon = e_1 \epsilon_1 + e_2 \epsilon_2$. Since $f$ is continuous at $x$, there exists a neighbourhood $V \subset X$ of $x$ such that  for each $y \in V,$ $$|f(y)- f(x)|_k < ' \epsilon.$$ Write  each $y \in V$ as $y = e_1 y_1 + e_2 y_2$ for some $y_1 \in e_1 V$ and $y_2 \in e_2 V$, we have   $$e_1 |f_1 (e_1 y - e_1 x)| + e_2 |f_2 (e_2 y - e_2 x)| < ' e_1 \epsilon_1 + e_2 \epsilon_2.$$ That is,  $$e_1 |f_1 (y_1 - x_1)| + e_2 |f_2 (y_2  - x_2 )| < ' e_1 \epsilon_1 + e_2 \epsilon_2,$$ for each $y_1 \in e_1 V$ and $y_2 \in e_2 V$. Then $e_1 V $ and $e_2 V$ are neighbourhoods of $x_1$ and $x_2$ in $e_1 X$ and $e_2 X$ respectively such that $$  |f_1 (y_1 - x_1)| <  \epsilon_1 $$ and $$|f_2 (y_2  - x_2 )|<  \epsilon_2, $$  for each $y_1 \in e_1 V$ and $y_2 \in e_2 V$. Thus,  ${f_1}_{|e_1 X}$  and ${f_2}_{| e_2 X}$ are  continuous.

Conversely, suppose  ${f_1}_{|e_1 X}$  and ${f_2}_{| e_2 X}$ are  continuous. Let $x \in  X$ and $\epsilon >' 0$ be given. Then there exist $x_1 \in e_1 X,  x_2 \in e_2 X$, $\epsilon_1, \epsilon_2>0$ such that $e_1 x= x_1,\; e_2 x= x_2$ and $\epsilon = e_1 \epsilon_1 + e_2 \epsilon_2$. Since  ${f_1}_{|e_1 X}$  and ${f_1}_{| e_2 X}$ are  continuous at $x_1$ and $x_2$ respectively, there exist neighbourhoods $e_1 U$ and $e_2 V$   of $x_1$ and $x_2$ in $e_1 X$ and $e_2 X$ respectively such that $$  |f_1 (y_1 - x_1)| <  \epsilon_1 $$ and $$|f_2 (y_2  - x_2 )|<  \epsilon_2, $$  for each $y_1 \in e_1 U$ and $y_2 \in e_2 V$. Then $O = e_1 U + e_2 V$ is a neighbourhood of $x$ in $X$ such that
 \begin{eqnarray*}
|f(y)- f(x)|_k &=& e_1 |f_1 (e_1 y - e_1 x)| + e_2 |f_2 (e_2 y - e_2 x)| \\
&=& e_1 |f_1 (y_1 - x_1)| + e_2 |f_2 (y_2  - x_2 )|  \\
&<'&  \epsilon_1 + \epsilon_2 \\
&=&  \epsilon,
\end{eqnarray*}
 for each $y \in O$. This shows that   $f$ is continuous.\\\\
$(ii)$  Suppose $f$ is $\mathbb{D}$-bounded in a neighbourhood $V\subset X$ of $0$. Then there exists a hyperbolic number $M >'0$ such that $|f(x)|_k \leq ' M$ for each $x \in V$. Thus $$ e_1| f_1 (e_1 x)| + e_2 |f_2 (e_2 x)| < ' M,$$ for each $x \in V$. Also, there exist $M_1 , M_2 >0$ such that  $M = e_1  M_1 + e_2 M_2$.  Then $ e_1 V$ and $ e_2 V$ are neighbourhoods of $0$  in $e_1 X$ and $e_2 X$ respectively such that  $$  |f_1 (x_1)|<M_1 $$ and  $$  |f_2 (x_2)|<M_2,$$ for each $x_1=e_1 x   \in  e_1 V$ and  $x_2 = e_2 x   \in  e_2 V.$ 

Conversely, suppose ${f_1}_{|e_1 X}$  and ${f_1}_{| e_2 X}$  are bounded in neighbourhoods $e_1 U$ and $e_2 V$   of $0$  in $e_1 X$ and $e_2 X$ respectively. Then there exist real numbers $M_1, M_2 >0$ such that  $$  |f_1 (x_1)|<M_1 $$ and  $$  |f_2 (x_2)|<M_2,$$ for each $x_1  \in  e_1 U$ and  $x_2   \in  e_2 V.$  Set  $M = e_1  M_1 + e_2 M_2$ and $O = e_1 U + e_2 V$. Then $M>'0$ and $O$ is a neighbourhood of $0$ in $X$ such that for each $x \in O$, we have
 \begin{eqnarray*}
| f(x)|_k &=& e_1 |f_1 ( e_1 x)| + e_2 |f_2 ( e_2 x)| \\
&<'&  e_1  M_1 + e_2 M_2 \\
&=& M.
\end{eqnarray*}

\end{proof}

The next theorem follows from  \cite [Theorem 1.17, Theorem 1.18 ] {rudin} and Theorem \ref{Thhb1}.
\begin{theorem}\label{Thhb2}
Let $X$ be a topological $\mathbb{BC}$-module (or $\mathbb{D}$-module) and $f$ be a $\mathbb{BC}$-linear (or $\mathbb{D}$-linear) functional on $X$. Then the following statements hold:
\begin{enumerate}
\item[(i)] $f$ is continuous at $0$.
\item[(ii)]  $f$ is continuous.
\item[(iii)]  $f$ is $\mathbb{D}$-bounded in some neighbourhood of $0$.
\end{enumerate}
\end{theorem}
\begin{theorem}\label{Thhb3}
 Let $X$ be a topological   $\mathbb{D}$-module and 
\begin{equation}\label{Eqhb2}
f=e_1{f_1} +e_2 {f_2}
\end{equation} be a  $\mathbb{D}$-linear functional on $X$, where  $f_1$ and $f_2 $   are  non-constant $\mathbb{R}$-linear  functionals  on $X$ and $A$ be a   $\mathbb{D}$-convex subset $X$. Then the following statements hold:
\begin{enumerate}
\item[(i)] $f(A)$ is   $\mathbb{D}$-convex.
\item[(ii)] If $A$ is open,  then so is $f(A)$.
\end{enumerate}
\end{theorem}
\begin{proof}$(i)$ Let $x, y \in f(A)$ and $\lambda \in \mathbb{D}$ such that $ 0 \leq' \lambda \leq' 1$. Then there exist $x',y' \in A$ such that $x= f(x')$ and $y = f(y')$. Since  $A$ is  $\mathbb{D}$-convex, $ \lambda x'+ (1- \lambda )  y' \in A.$  By $\mathbb{D}$-linearity of $f$, we  have  
 \begin{eqnarray*}
\lambda x + (1- \lambda )y &=& \lambda f(x') + (1- \lambda )  f(y')\\
 &=& f( \lambda x') +f( (1- \lambda )  y')\\
 &=& f( \lambda x'+ (1- \lambda )  y')\\
 &\in &  f(A).
\end{eqnarray*}\\
$(ii)$ By $(i)$,   $f(A)$ is    $\mathbb{D}$-convex, so we can write  $f(A) =  e_1f( A) + e_2 f( A)$. Now, by  $\mathbb{D}$-linearity  of $f$, and Equation (\ref{Eqhb2}), we get $e_1 f(A) =e_1 f(e_1 A)= e_1 f_1 (e_1 A)$ and  $e_2 f(A) = e_2 f(e_2 A)= e_2 f_2 (e_2 A)$. Therefore, we have  $ f(A) =e_1 f_1 (e_1 A) + e_2 f_2 (e_2 A).$  Since  $e_1A$ and $e_2 A$ are open sets in  $e_1 X$ and $e_2 X$ respectively and any non-constant  $\mathbb{R}$-linear functional on a  $\mathbb{R}$-vector space is an open mapping, it follows that $ f_1 (e_1 A)$ and $ f_2 (e_2 A)$ are open sets in $\mathbb{R}$. Therefore,  $f(A) =e_1 f_1 (e_1 A) + e_2 f_2 (e_2 A)$ is an open set in  $\mathbb{D}$.
\end{proof}
\end{section} 


\begin{section}{Uniform Boundedness Principle}

In this section, we study the  Baire category theorem for  $\mathbb{D}$-metric spaces and the principle of  uniform boundedness for $F$- bicomplex and  $F$-hyperbolic modules.

\begin{definition}\cite[Definition 5.6] {HS}
Let   $d_{\mathbb{D}} :  X \times X \rightarrow  \mathbb{D}$ be a function such that for any $x, y, z \in X$, the following properties hold:
 \begin{enumerate}
\item[(i)] $d_{\mathbb{D}}(x, y) \geq ' 0$  and  $d_{\mathbb{D}}(x, y) = 0$ if and only if $x =y$,
\item[(ii)]  $ d_{\mathbb{D}}(x, y) = d_{\mathbb{D}}(y, x$),
\item[(iii)] $d_{\mathbb{D}}(x, z) \leq' d_{\mathbb{D}}(x, y) + d_{\mathbb{D}}(y, z)$.
\end{enumerate}
Then $d_{\mathbb{D}}$ is called a hyperbolic-valued (or  $\mathbb{D}$-valued) metric on $X$ and the pair $(X, d_{\mathbb{D}})$ is called  a   hyperbolic metric (or $\mathbb{D}$-metric) space.
\end{definition}

\begin{definition} Let  $(X, d_{\mathbb{D}})$ be a  $\mathbb{D}$-metric space, $x \in X$ and $r>'0$ be a  hyperbolic number. Then    an open ball in $X$ with center $x$ and radius $r$ is denoted by $B(x, r)$ and is defined as $$B(x, r)  = \{y \in X\;:\;d_{\mathbb{D}}(x, y)<' r \}. $$
\end{definition}

\begin{definition} Let  $(X, d_{\mathbb{D}})$ be a  $\mathbb{D}$-metric space and $M \subset X$. Then  a point $x \in M$ is called an interior point of $M$  if there exists a $r>'0$ such that $B(x, r) \subset M.$
\end{definition}

\begin{definition} Let  $(X, d_{\mathbb{D}})$ be a  $\mathbb{D}$-metric space and $M \subset X$. Then  a point $x \in M$ is called a limit point of $M$  if for every  $r>'0$, the open ball $B(x, r)$ contains a point of $M$ other than $x$.
\end{definition}

\begin{definition} Let  $(X, d_{\mathbb{D}})$ be a  $\mathbb{D}$-metric space and $M \subset X$. Then $M$ is said to be open in $X$  if every point of $M$ is an interior point.
\end{definition}

\begin{definition} Let  $(X, d_{\mathbb{D}})$ be a  $\mathbb{D}$-metric space and $M \subset X$. Then $M$ is said to be closed in $X$  if it contains all its limit points.
\end{definition}

\begin{definition} A sequence $\{ x_n\}$ in a $\mathbb{D}$-metric space $(X, d_{\mathbb{D}})$ is said to converge to a point $x \in X$  if for every $\epsilon >' 0$, there exists $ N \in \mathbb{N}$ such that $$  d_{\mathbb{D}}(x_n , x) <' \epsilon \;\;\; \textmd{for every}\;\; n \geq N.$$ 
\end{definition}
\begin{definition} A sequence $\{ x_n\}$ in a $\mathbb{D}$-metric space $(X, d_{\mathbb{D}})$ is said to be a Cauchy sequence  if for every $\epsilon >' 0$, there exists $ N \in \mathbb{N}$ such that $$  d_{\mathbb{D}}(x_n , x_m) <' \epsilon \;\;\; \textmd{for every}\;\; n, m \geq N.$$
\end{definition}
\begin{definition}
 A $\mathbb{D}$-metric space $X$ is said to be complete if every Cauchy sequence in $X$ converges in $X$.
\end{definition}
The follwing theorem can be proved easily
\begin{theorem} A  $\mathbb{D}$-metric space is a topological space.
\end{theorem}

The following theorem is the Baire category theorem for  $\mathbb{D}$-metric spaces.
\begin{theorem}\label{BCT}
Let $(X, d_{\mathbb{D}})$ be a complete $\mathbb{D}$-metric space such that $X= \bigcup\limits_{n=1}^{\infty}F_n$ where each $F_n$ is a closed subset of $X$. Then at least one of the $F_n$'s contains a non-empty open set.
\end{theorem}
\begin{proof}
Suppose each $F_n$  does not contain a non-empty open set. Since $X\setminus F_1$ is a non-empty open set, we choose $x_1 \in X\setminus F_1$ and $0<'\epsilon_1 <' 1/2$ such that  $B_1\cap F_1= \phi$ where $B_1 =B(x_1, \epsilon_1)$. The open ball $B(x_1, \epsilon_1/2)$ is not contained in $F_2$, hence there is some $x_2 \in B(x_1, \epsilon_1/2)$ and $0<' \epsilon_2 <' 1/2^2 $ such that $B_2\cap F_2 = \phi$ and $B_2 \subset B(x_1, \epsilon_1/2)$, where $B_2 =B(x_2, \epsilon_2)$. By induction, we obtain the sequences  $\{x_n\}$ and $\{\epsilon_n\}$ which satisfy the following:
\begin{enumerate}
\item[(i)] $B_{n+1}\subset B(x_n, \epsilon_n/2)$,
\item[(ii)] $0<' \epsilon_n <' 1/2^n $ and
\item[(iii)]$ B_n  \cap F_n = \phi$.
\end{enumerate}
For $n<m$,
\begin{eqnarray*}
d_{\mathbb{D}}(x_n, x_m )& \leq ' & d_{\mathbb{D}}(x_n, x_{n+1} )+d_{\mathbb{D}}(x_{n+1}, x_{n+2} )+\;.\;.\;.+d_{\mathbb{D}}(x_{m-1}, x_m)\\
&< '& \epsilon_n/2+ \epsilon_{n+1}/2+\;.\;.\;.+\epsilon_{m-1}/2\\
&< '& 1/2^{n+1}+ 1/2^{n+2}+\;.\;.\;.+1/2^{m}\\
&< '& 1/2^n.
\end{eqnarray*}
Thus,   $\{x_n\}$ is a Cauchy sequence in $X$. Since $X$ is complete, $\{x_n\}$ converges to a point $x$. Also, 
\begin{eqnarray*}
d_{\mathbb{D}}(x_n, x )& \leq ' & d_{\mathbb{D}}(x_n, x_m )+d_{\mathbb{D}}(x_{m}, x )\\
&< '& \epsilon_n/2+d_{\mathbb{D}}(x_{m}, x ) \\
&\rightarrow& \epsilon_n/2.
\end{eqnarray*}
It follows that $x \in B_n$,  for each $n$. Since $ B_n \cap F_n = \phi$, we see that $x \notin F_n$,  for each $n$. Therefore, $x \notin \bigcup\limits_{n=1}^{\infty}F_n$. This contradicts the fact that  $X= \bigcup\limits _{n=1}^{\infty}F_n$. Thus,  one of the $F_n$'s contains a non-empty open set. 
\end{proof}

\begin{definition}
An $F$-bicomplex module  (or $F$-$\mathbb{BC}$  module) is a $\mathbb{BC}$-module $X$ having a $\mathbb{D}$-valued metric $d_{\mathbb{D}}$ such that the following properties hold:
\begin{enumerate}
\item[(i)] $d_{\mathbb{D}}$ is translation invariant.
\item[(ii)] $ \cdot\; : \mathbb{BC} \times X \longrightarrow X$ is continuous.
\item[(iii)] $(X, d_{\mathbb{D}})$ is  complete $\mathbb{D}$-metric space.
\end{enumerate}
\end{definition}

\begin{definition}
An $F$-hyperbolic module (or $F$-$\mathbb{D}$  module) is a $\mathbb{D}$-module $X$ having a $\mathbb{D}$-valued metric $d_{\mathbb{D}}$ such that the following properties hold:
\begin{enumerate}
\item[(i)] $d_{\mathbb{D}}$ is translation invariant.
\item[(ii)] $ \cdot\; : \mathbb{D} \times X \longrightarrow X$ is continuous.
\item[(iii)] $(X, d_{\mathbb{D}})$ is  complete $\mathbb{D}$-metric space.
\end{enumerate}
\end{definition}

For each $x\in X$, define $||x||_{\mathbb{D}}=d_{\mathbb{D}}(x, 0)$. We see that, if  $d_{\mathbb{D}}$ is translation invariant  $\mathbb{D}$-metric on $X$, then the properties $d_{\mathbb{D}}(x, y)=0$ if and only if $x=y$, $d_{\mathbb{D}}(x, y) \leq ' d_{\mathbb{D}}(x, z)+d_{\mathbb{D}}(z, y)$  and  $d_{\mathbb{D}}(x, y)=d_{\mathbb{D}}(y, x)$  are equivalent to $||x||_{\mathbb{D}}=0$ if and only if $x=0$,  $||x+y||_{\mathbb{D}} \leq'  ||x||_{\mathbb{D}}+ ||y||_{\mathbb{D}}$  and $||-x||_{\mathbb{D}} =||x||_{\mathbb{D}}$ respectively. Infact,\\\\
$(i)$
\begin{eqnarray*}
 ||x||_{\mathbb{D}}=0 & \Leftrightarrow & d_{\mathbb{D}}(x, 0) =0\\
& \Leftrightarrow & x=0.
\end{eqnarray*}
$(ii)$ 
\begin{eqnarray*}
 ||x+y||_{\mathbb{D}} &= & d_{\mathbb{D}}(x+y, 0)\\
 &= & d_{\mathbb{D}}(x, -y) \\
 &\leq' & d_{\mathbb{D}}(x, 0)+ d_{\mathbb{D}}(0, -y)\\
&= & d_{\mathbb{D}}(x, 0)+d_{\mathbb{D}}(y, -y+y)\\
&= & d_{\mathbb{D}}(x, 0)+d_{\mathbb{D}}(y, 0)\\
 &= &||x||_{\mathbb{D}}+||y||_{\mathbb{D}}.
\end{eqnarray*}
$(iii)$ 
\begin{eqnarray*}
 ||-x||_{\mathbb{D}} &= & d_{\mathbb{D}}(-x, 0)\\
 &= & d_{\mathbb{D}}(0, -x) \\
 &= & d_{\mathbb{D}}(x, -x+x) \\
&= & d_{\mathbb{D}}(x, 0)\\
 &= &||x||_{\mathbb{D}}.
\end{eqnarray*}

\begin{theorem}\label{URT} An $F$-$\mathbb{BC}$ module is a topological $\mathbb{BC}$-module.
\end{theorem}

The next theorem demonstrates the principle of uniform boundedness in the setting of  $F$-$\mathbb{BC}$ modules. The  uniform boundedness principle for $F$-module spaces with real-valued metric was stated in \cite{LL}. Here, we prove the result for $F$-$\mathbb{BC}$ modules with hyperbolic-valued metric.
\begin{theorem}\label{UBP}
For each $\alpha \in \wedge,$ let $T_\alpha : X \rightarrow Y$ be a continuous $\mathbb{BC}$-linear map form a $F$-$\mathbb{BC}$ module $X$ to a $F$-$\mathbb{BC}$ module $Y$. If for each $x \in X,$ the set $ \{ T_\alpha( x) : \alpha \in \wedge  \}$ is $\mathbb{D}$-bounded then $\lim_{x \rightarrow 0} T_\alpha (x )= 0$ uniformly for $\alpha \in \wedge.$
\end{theorem}
\begin{proof} For given $\epsilon >' 0$ and each positive integer $n$, consider the set $$X_n = \{ x \in X \; : \;  \left |\left |\frac{1}{n}T_{\alpha}(x)\right|\right|_{\mathbb{D}}+  \left |\left |\frac{1}{n}T_{\alpha}(-x)\right|\right|_{\mathbb{D}} \leq ' \frac{\epsilon}{2},\;\; \alpha \in \wedge \}. $$ Since $T_\alpha$ is continuous we see that each set $X_n$ is closed. Also, $\cup_{n=1}^{\infty}X_n=X$. Therefore, by Theorem \ref{BCT}, some $X_{n_{0}}$ contains an open ball $B(x_0 , \delta)$. That is, if $||x||_{\mathbb{D}}  <' \delta$, then $$ \left |\left |\frac{1}{n_{0}}T_{\alpha}(x_0 +x)\right|\right|_{\mathbb{D}} \leq ' \frac{\epsilon}{2}.$$ Since $T_\alpha$ is $\mathbb{BC}$-linear, we have
\begin{eqnarray*}
 \left |\left |\frac{1}{n_{0}}T_{\alpha}(x)\right|\right|_{\mathbb{D}} &= &  \left |\left |\frac{1}{n_{0}}T_{\alpha}(x+x_0 -x_0 )\right|\right|_{\mathbb{D}}\\
&= &  \left |\left |\frac{1}{n_{0}}T_{\alpha}(x+x_0 )+T_{\alpha}( -x_0 )\right|\right|_{\mathbb{D}}\\
 &\leq' & \left |\left |\frac{1}{n_{0}}T_{\alpha}(x+x_0 )\right|\right|_{\mathbb{D}}+  \left |\left |\frac{1}{n_{0}}T_{\alpha}(-x_0 )\right|\right|_{\mathbb{D}}.\\
\end{eqnarray*}
Thus, if  $||x||_{\mathbb{D}}  <' \delta$, then $$\left |\left |T_{\alpha}\left(\frac{1}{n_{0}} x\right)\right|\right|_{\mathbb{D}}=  \left |\left |\frac{1}{n_{0}}T_{\alpha}(x)\right|\right|_{\mathbb{D}} \leq ' \epsilon.$$ Since the mapping $ x \rightarrow x/n_{0}$ is a homeomorphism, it follows that $\lim_{x \rightarrow 0} T_\alpha (x )= 0$ uniformly for $\alpha \in \wedge.$
\end{proof}
Theorems \ref{URT} and \ref{UBP} in the setting of $F$-$\mathbb{D}$ modules can be stated as:

\begin{theorem} An $F$-$\mathbb{D}$ module is a topological $\mathbb{D}$-module.
\end{theorem}
\begin{theorem}
For each $\alpha \in \wedge,$ let $T_\alpha : X \rightarrow Y$ be a continuous $\mathbb{D}$-linear map form a $F$-$\mathbb{D}$ module $X$ to a $F$-$\mathbb{D}$ module $Y$. If for each $x \in X,$ the set $ \{ T_\alpha( x) : \alpha \in \wedge  \}$ is $\mathbb{D}$-bounded then $\lim_{x \rightarrow 0} T_\alpha (x )= 0$ uniformly for $\alpha \in \wedge.$
\end{theorem}

\end{section}


\begin{section}{Open Mapping and Closed Graph Theorems}
In this section, we discuss the open mapping theorem, the inverse mapping theorem and the closed graph theorem for  $F$- bicomplex and  $F$-hyperbolic modules.   The open mapping  and the closed graph theorems  were stated in \cite{LL} for $F$-module spaces with real-valued metric. The proofs of the following two results are on similar lines as in  \cite[Theorem 1, pp-55] {dusc}.
\setcounter{equation}{0}

\begin{lemma}\label{OMTL}
Let  $T : X \rightarrow Y$  be a   continuous $\mathbb{BC}$-linear map from a $F$-$\mathbb{BC}$ module  $X$ onto  a $F$-$\mathbb{BC}$ module  $Y$. Then the image of a neighbourhood of the origin in $X$ under $T$ contains  a neighbourhood of the origin in $Y$.
\end{lemma}
\begin{proof} We prove the lemma in two steps. In Step I we show that  the closure of the image of a neighbourhood of the origin  under $T$ contains  a neighbourhood of the origin. And in Step II we finally prove that the image of a neighbourhood of the origin under $T$ contains  a neighbourhood of the origin. \\
$\underline{Step\; I}$: Let $G$ be a neighbourhood of $0$ in $X$. Since $X$ is a topological $\mathbb{BC}$-module, any algebraic combination of variables $x_1 , x_2$ is continuous as a map from $X \times X$ into $X$, so  there is a neighbourhood $V$ of $0$ in $X$ such that $V-V \subset G$. We see that for each $x \in X$ there exists $n \in \mathbb{N}$ such that $x \in n V$. We, therefore, have $X= \bigcup\limits_{n=1}^{\infty}n V$. Since $T$ is surjective and  $\mathbb{BC}$-linear it follows that $$Y= T(X) = T \left ( \bigcup_{n=1}^{\infty}n V \right )= \bigcup_{n=1}^{\infty}n T V  =  \bigcup_{n=1}^{\infty}\overline{n T V}.$$ By Theorem \ref{BCT}, one of the sets $\overline{n T V}$ contains a non-empty open set. This implies that $\overline{ T V}$ also contains a non-empty open set, say, $U$. It follows that, $$ \overline{ T G} \supseteq \overline{ T V-TV} \supseteq \overline{ T V}-  \overline{TV}  \supseteq U-U.$$ Clearly, the set $u- U$ is open, so $U-U =  \bigcup\limits_{u \in U}(u-U)$ is an open set containing 0. Denote the set  $U-U$ by $W$. Then $W$ is a neighbourhood of $0$ in $Y$ such that $W \subset \overline{ T G}.$\\
$\underline{Step\; II}$:
For any $\epsilon>'0$, denote $$X_{\epsilon}= \{ x \in X \; :\; ||x||_{\mathbb{D}} <' \epsilon\},$$ and $$ Y_{\epsilon}= \{ y \in Y \; :\; ||y||_{\mathbb{D}} <' \epsilon\}.$$  Let $\epsilon_{0}>' 0$ be arbitrary, and let $\{ \epsilon_{i}\}_{i=1}^{\infty}$ be a sequence of positive hyperbolic numbers such that $\sum\limits_{i=1}^{\infty} \epsilon_{i}<' \epsilon_{0}.$ By Step I, there is  a sequence $\{ \eta_{i}\}_{i=0}^{\infty}$ of positive hyperbolic numbers tending to $0$ such that 
\begin{equation}\label{OMTL1}
Y_{\eta_{i}} \subset  \overline{ T X_{\epsilon_{i}}}, \;\;\;\;\;\;\;\;\;\;\;\;\;\;\; i=0,\;1\;2,\;.\;.\;.\;.
\end{equation}
Let $y \in Y_{\eta_{0}}$. Then, from (\ref{OMTL1}), with $i=0$, there is an $x_{0} \in  X_{\epsilon_{0}}$ such that $||y-Tx_{0}||_{\mathbb{D}} <' \eta_{1}.$ Thus, $y-Tx_{0} \in Y_{\eta_{1}}$. Again,  from (\ref{OMTL1}), with $i=1$, there is an $x_{1} \in  X_{\epsilon_{1}}$ such that $||y-Tx_{0}-Tx_{1}||_{\mathbb{D}} <' \eta_{2}.$ In $(n+1)$th step we find an $x_{n}$ such that   \begin{equation}\label{OMTL2}
||y-\sum_{i=0}^{n}Tx_{i}||_{\mathbb{D}} <' \eta_{n+1}, \;\;\;\;\;\;\;\;\;\;\;\;\;\;\; n=0,\;1\;2,\;.\;.\;.\;.
\end{equation}
Let $z_n = x_0 + x_1 + x_2 + \;. \;.\;.\;+x_n$. For, $n>m$, we have
 $$ ||z_n - z_m||_{\mathbb{D}} \leq' \sum\limits_{i=m+1}^{n}||x_i||_{\mathbb{D}} <  \sum\limits_{i=m+1}^{\infty} \epsilon_{i} <' \epsilon_{0}.$$ Thus, $\{z_n\}$ is a Cauchy sequence in $X$ and it converges to a point, say, $x \in X$. Then, $$  ||x||_{\mathbb{D}} = \lim_{n \rightarrow \infty} ||z_n ||_{\mathbb{D}} \leq ' \lim_{n \rightarrow \infty} \sum\limits_{i=0}^{n} \epsilon_{i} <' 2 \epsilon_{0}.$$
This shows that $x \in X_{2\epsilon_{0}}$. Also, since $T$ is continuous, $Tz_n \rightarrow Tx$ and from (\ref{OMTL2}), we see that  $Tx=y$. Thus we have shown that $Y_{\eta_{0}} \subset TX_{2\epsilon_{0}}.$ 
\end{proof} 
We now prove the  open mapping theorem for  $F$-$\mathbb{BC}$ modules. 
\begin{theorem}\label{OMT}
Let  $T : X \rightarrow Y$  be a   continuous $\mathbb{BC}$-linear map from a $F$-$\mathbb{BC}$ module  $X$ onto  a $F$-$\mathbb{BC}$ module  $Y$. Then $T$  is an open map.
\end{theorem}
\begin{proof}
Let $G$ be a non-empty open set in $X$. To show that $TG$ is open in $Y$, let $y=Tx \in TG$. Since $G$ is open, there exists a neighbourhood $U$ of $0$ in $X$ such that $x + U \subset G$. By Lemma \ref{OMTL}, there is a neighbourhood $V$ of $0$ in $Y$ such that $ V \subset TU$. Then $y+V$ is a neighbourhood of $y$ such that $$ y+ V =Tx+V \subset  Tx + TU =  T(x+U) \subset TG.$$ This shows that $TG$ is open in $Y$ and hence $T$ is an open map.
\end{proof}

The next result is the inverse mapping theorem for  $F$-$\mathbb{BC}$ modules.
\begin{theorem} \label{IMT}
Let  $T : X \rightarrow Y$  be a   continuous bijective $\mathbb{BC}$-linear map from a $F$-$\mathbb{BC}$ module  $X$ to  a $F$-$\mathbb{BC}$ module  $Y$   Then $T$ has a continuous $\mathbb{BC}$-linear inverse.
\end{theorem} 
\begin{proof}
To show that $T^{-1}$ is $\mathbb{BC}$-linear, let $y_1, y_2 \in Y$ and $x_1, x_2 \in X$ such that $T(x_1)=y_1$ and $T(x_2)=y_2$. Then $$ T(x_1 +x_2)= T(x_1)+T(x_2)= y_1 + y_2.$$ Therefore, $$T^{-1}( y_1 + y_2) =x_1 +x_2 = T^{-1}( y_1) + T^{-1}( y_2).$$ Also, for any $\lambda \in \mathbb{BC}$, we have, $$T(\lambda x_1 )= \lambda T(x_1) =\lambda y_1. $$
Thus, $$ T^{-1}( \lambda y_1) = \lambda x_1=\lambda T^{-1}( y_1).$$ Now to show that $T^{-1}$ is continuous, it is enough to show that the inverse image of an open set under $T^{-1}$ is open. So let $M \subset X$ be an open set. By Theorem \ref{OMT}, $(T^{-1})^{-1}=T$ maps open sets onto open sets, hence  $(T^{-1})^{-1}(M)$ is open in $Y$. This completes the proof.

\end{proof}

\begin{definition}
 Let $X$ and $Y$ be two $F$-$\mathbb{BC}$ modules.
Let $T$ be a $\mathbb{BC}$-linear map  whose domain $D(T)$ defined as
$$
D(T)= \{ x \in X : Tx \in Y \}
$$
is a $\mathbb{BC}$-submodule in $X$ and whose range lies in $Y$. Then the graph of $T$ is the set of all points in $X \times Y$ of the form 
$( x , \; Tx )$ with $x \in D(T).$ 
\end{definition}
A $\mathbb{BC}$-linear operator $T$ is said to be closed if its graph is closed in the product space $X \times Y.$ An equivalent statement is 
as follows: \\

A $\mathbb{BC}$-linear operator $T$ is closed if whenever $x_n \in D(T), \; x_n \longrightarrow x, \; T x_n \longrightarrow y \Rightarrow  x \in D(T) $ and $Tx =y.$ 

Note that the product $X \times Y$ of two $F$-$\mathbb{BC}$ modules $X$ and $Y$  is also an $F$-$\mathbb{BC}$ module. \\

The next result is closed graph theorem in the setting of  $F$-$\mathbb{BC}$ modules.

\begin{theorem} \label{CGT}
Let $X$ and $Y$ be $F$-$\mathbb{BC}$ modules and $T : X \rightarrow Y$  be a  closed $\mathbb{BC}$-linear map.  Then $T$ is continuous.
\end{theorem}
\begin{proof}
Clearly the graph $G$ of $T$ is a closed $\mathbb{BC}$-submodule in the product $F$-$\mathbb{BC}$ module $X \times Y$, hence $G$ is a complete $\mathbb{D}$-metric space. Thus $G$ is an $F$-$\mathbb{BC}$ module. The map $p_X : (x , \; Tx) \mapsto x $ of $G$ onto $X$ is one-to-one, $\mathbb{BC}$-linear, and continuous. Hence, by Theorem \ref{IMT}, its inverse $p_{X}^{-1}$ is continuous. Thus $T= p_Y p_{X}^{-1}$ is continuous.
\end{proof}
Theorems \ref{OMT}, \ref{IMT} and \ref{CGT}  in the setting of $F$-$\mathbb{D}$ modules have been stated below:
\begin{theorem}
Let  $T : X \rightarrow Y$  be a   continuous $\mathbb{D}$-linear map from a $F$-$\mathbb{D}$ module  $X$ onto  a $F$-$\mathbb{D}$ module  $Y$. Then $T$  is an open map.
\end{theorem}
\begin{theorem} 
Let  $T : X \rightarrow Y$  be a   continuous bijective $\mathbb{D}$-linear map from a $F$-$\mathbb{D}$ module  $X$ to  a $F$-$\mathbb{D}$ module  $Y$   Then $T$ has a continuous $\mathbb{D}$-linear inverse.
\end{theorem} 

\begin{theorem}
Let $X$ and $Y$ be $F$-$\mathbb{D}$ modules and $T : X \rightarrow Y$  be a  closed $\mathbb{D}$-linear map.  Then $T$ is continuous.
\end{theorem}
\end{section}

\begin{section}{Hahn Banach Separation Theorem for Topological  Hyperbolic Modules } 
In this section, we present a proof of the Hahn Banach separation theorem for topological  hyperbolic modules.
\begin{theorem}\label{Thhb4}
Let $X$ be a topological $\mathbb{D}$-module and  $A, B$ be non-empty $\mathbb{D}$-convex subsets of $X$ such that $e_1 A \cap e_1 B = e_2 A \cap e_2 B = \emptyset.$ If $A$ is open, then there exist a continuous $\mathbb{D}$-linear functional $f$ on $X$ and $\gamma \in \mathbb{D}$ such that  $$ f(a) <' \gamma \leq' f(b),$$ for each $a \in A$ and $b \in B.$
\end{theorem}
\begin{proof}  Choose $a_0 \in A$ and $b_0 \in B$. Take $G= A-B+x_0$, where $x_0= b_0-a_0$. Then $G$  is an open set in $X$. Clearly, $0 \in G$. Therefore,  $G$ is  $\mathbb{D}$-absorbing in $X$.  Let $x, y \in G$ and $0 \leq' \lambda \leq' 1$. Then, there exist $a_1, a_2 \in A$ and $b_1, b_2 \in B$ such that $x= a_1 - b_1 + x_0$ and $y =  a_2 - b_2 + x_0$. Since $A$ and $B$ are $\mathbb{D}$-convex, 
\begin{eqnarray*}
\lambda x + (1- \lambda)y &=& \lambda ( a_1 - b_1 + x_0)  +  (1- \lambda) ( a_2 - b_2 + x_0)\\
 &=& \lambda a_1  +  (1- \lambda) a_2  +  \lambda b_1  +  (1- \lambda) b_2  + x_0 \in G.
\end{eqnarray*}
Thus, $G$ is a  $\mathbb{D}$-convex, $\mathbb{D}$-absorbing set containing $0$.  Let $q_G$ be the $\mathbb{D}$-Minkowski functional of $G$. Then, there exist real valued Minkowski functionals $$q_{e_1 G} : e_1 X \longrightarrow \mathbb{R}$$ and $$q_{e_2 G} : e_2 X \longrightarrow \mathbb{R}$$ such that $q_G(x) = q_{e_1 G}( e_1 x) e_1 +  q_{e_2 G}( e_2 x) e_2$, for each $x \in X$. Since $e_1 A \cap e_1 B = e_2 A \cap e_2 B = \emptyset$, it follows that $e_1 x_0 \notin e_1 G$ and $e_2 x_0 \notin e_2 G$. Thus, $ q_{e_1 G}( e_1 x_0) \geq 1$ and  $ q_{e_2 G}( e_2 x_0) \geq 1$, which yields  $q_G (x_0) \geq' 1$. Take $Y = \{  \lambda x_0 \;  : \; \lambda \in \mathbb{D}\}$ and define a  $\mathbb{D}$-linear functional $g : Y  \longrightarrow \mathbb{D}$ by $$  g(  \lambda x_0) = \lambda, \;\;\; \textmd{for each} \; \lambda  \in \mathbb{D}.$$ We now show that $g$ is dominated by $q_G$ on $\mathbb{D}$-submodule $Y$ of $X$. For this, write $\lambda \in \mathbb{D}$ as $\lambda = e_1 \lambda_1 +  e_2 \lambda_2$, where $ \lambda_1,  \lambda_2 \in \mathbb {R}$. Then, we have the following cases:\\
Case $(i)$: If $\lambda_1 \geq 0$ and $\lambda_2 \geq 0$, then 
\begin{eqnarray*}
 g(  \lambda x_0) &=&  e_1 \lambda_1 +  e_2 \lambda_2 \\
& \leq' & e_1 \lambda_1  q_{e_1 G}( e_1 x_0)  +  e_2 \lambda_2  q_{e_2 G}( e_2 x_0)  \\
& = &  e_1  q_{e_1 G}( e_1 \lambda_1  x_0)  +  e_2 q_{e_2 G}( e_2  \lambda_2  x_0)  \\
& = &  q_G (\lambda x_0).
\end{eqnarray*}
Case $(ii)$:  If $\lambda_1 < 0$ and $\lambda_2 \geq 0$, then 
\begin{eqnarray*}
 g(  \lambda x_0) &=&  e_1 \lambda_1 +  e_2 \lambda_2 \\
& \leq' & e_1 0 +  e_2 \lambda_2  q_{e_2 G}( e_2 x_0) \\
& \leq' &  e_1  q_{e_1 G}( e_1 \lambda_1  x_0)  +  e_2 q_{e_2 G}( e_2  \lambda_2  x_0)  \\
& = &  q_G (\lambda x_0).
\end{eqnarray*}
Case $(iii)$:  If $\lambda_1 \geq  0$ and $\lambda_2 < 0$, then 
\begin{eqnarray*}
 g(  \lambda x_0) &=&  e_1 \lambda_1 +  e_2 \lambda_2 \\
& \leq' & e_1  \lambda_1  q_{e_1 G}( e_1 x_0) +  e_2 0 \\
& \leq' &  e_1  q_{e_1 G}( e_1 \lambda_1  x_0)  +  e_2 q_{e_2 G}( e_2  \lambda_2  x_0)  \\
& = &  q_G (\lambda x_0).
\end{eqnarray*}
Thus, $g \leq' q_G$ on $Y$. By  \cite [Theorem 5] {Hahn}, there exists a  $\mathbb{D}$-linear functional $f$ on $X$ such that $$ f_{|Y}= g \; \textmd{ and}\; \; f \leq' q_G.$$  Now, $ q_G  \leq' 1$ on $G$ implies that  $ f  \leq' 1$ on $G$. From this, it follows that  $ f  \geq'  -1$ on $-G$. Therefore, we have $|f|_k \leq' 1$ on $  G \cap (-G)$.  Since  $ G \cap (-G)$ is an open set containing $0$, by Theorem \ref{Thhb2},  $f$ is continuous on $X$. Now, let $a \in A$ and $b \in B$. Since $a-b+x_0 \in G$ and $G$ is open, we see that
\begin{eqnarray*}
f(a)-f(b)+1 &=& f(a-b+x_0)\\
&\leq' & q_G (a-b+x_0)\\
&<' & 1,
\end{eqnarray*} 
and so $f(a) <' f(b)$.  From Theorem \ref{Thhb3},  $f(A)$ is open in $ \mathbb{D}$.  Set $\gamma = \inf_{\mathbb{D}} f(B)$, we have, for each $a \in A$, $b \in B$,  $$ f(a) <' \gamma \leq' f(b).$$ 
\end{proof}

\end{section} 


\begin{section} { Hahn Banach Separation Theorem for Topological  Bicomplex Modules } 
 In this section, a generalized version of Hahn Banach separation theorem for topological  bicomplex modules is  proved. 

Let $X$ be a  $\mathbb{BC}$-module and $h$ be a $\mathbb{BC}$-functional on $X$. Then for  each $x \in X$, $$h(x) = g_1 (x)+ i g_2 (x)+ j g_3 (x)+ k g_4 (x), $$ where $g_1, g_2,  g_3$ and  $g_4 $ are $\mathbb{R}$-functionals on $X$. We now define a $\mathbb{D}$-functional on $X$ with the help of $h$ as follows:
For each $x \in X$, define $$h_{\mathbb{D}} (x) =  g_1 (x) + k  g_4 (x).$$  The following two theorems then follow from \cite[Section 4]{Hahn}:
\begin{theorem}
Let $X$ be a  $\mathbb{BC}$-module. Then the following statements hold:
\begin{enumerate}
\item[(i)] Let $h$ be a $\mathbb{BC}$-linear functional on $X$.  Then $h_{\mathbb{D}}$ is a $\mathbb{D}$-linear functional on $X$ and  $$ h(x) =h_{\mathbb{D}}(x)- ih_{\mathbb{D}}(ix), \;\;\; \textmd{for each}\;  x \in X.$$ 
\item[(ii)] Let $f$ be a $\mathbb{D}$-linear functional on $X$ and $h: X \longrightarrow \mathbb{BC}$ be defined by  $$ h(x) =f(x)- if(ix), \;\;\; \textmd{for each}\;  x \in X.$$  Then $h$ is a $\mathbb{BC}$-linear functional on $X$. 
\end{enumerate}
\end{theorem}

\begin{theorem}
Let $X$ be a  $\mathbb{BC}$-module. Then the following statements hold:
\begin{enumerate}
\item[(i)] Let $h$ be a $\mathbb{BC}$-linear functional on $X$.  Then $h_{\mathbb{D}}$ is a $\mathbb{D}$-linear functional on $X$ and  $$ h(x) =h_{\mathbb{D}}(x)- jh_{\mathbb{D}}(jx), \;\;\; \textmd{for each}\;  x \in X.$$ 
\item[(ii)] Let $f$ be a $\mathbb{D}$-linear functional on $X$ and $h: X \longrightarrow \mathbb{BC}$ be defined by  $$ h(x) =f(x)- jf(jx), \;\;\; \textmd{for each}\;  x \in X.$$  Then $h$ is a $\mathbb{BC}$-linear functional on $X$. 
\end{enumerate}
\end{theorem}
\begin{corollary}
Let $X$ be a topological $\mathbb{BC}$-module. If  a   $\mathbb{BC}$-linear functional $h$   on $X$  is continuous then so is  $h_{\mathbb{D}}$; and for  every continuous $\mathbb{D}$-linear functional $f$ on $X$, there exists a unique continuous $\mathbb{BC}$-linear functional $h$ on $X$ such that  $f =h_{\mathbb{D}}$.
\end{corollary}
\begin{theorem}\label{Thhb5}
Let $X$ be a topological $\mathbb{BC}$-module and  $A, B$ be non-empty $\mathbb{D}$-convex subsets of $X$ such that $e_1 A \cap e_1 B = e_2 A \cap e_2 B = \emptyset.$ If $A$ is open, then there exist a continuous $\mathbb{BC}$-linear functional $h$ on $X$ and $\gamma \in \mathbb{D}$ such that  $$h_{\mathbb{D}}(a) <' \gamma \leq' h_{\mathbb{D}}(b),$$ for each $a \in A$ and $b \in B.$
\end{theorem}
\begin{proof} By Theorem \ref{Thhb4}, there exist a continuous $\mathbb{D}$-linear functional $f$ on $X$ and $\gamma \in \mathbb{D}$ such that  $$ f(a) <' \gamma \leq' f(b),$$ for each $a \in A$ and $b \in B.$ Define a functional  $h: X \longrightarrow \mathbb{BC} $ by $$ h(x) =f(x)- if(ix), \;\;\; \textmd{for each}\;  x \in X.$$ Then $h$ is a continuous $\mathbb{BC}$-linear functional on $X$ such that for each $x \in X,$  $h_{\mathbb{D}}(x)= f(x).$ It follows that  $$h_{\mathbb{D}}(a) = f(a) <' \gamma \leq' f(b) = h_{\mathbb{D}}(b),$$ for each $a \in A$ and $b \in B.$
\end{proof}
\end{section} 
\begin{section}{Another  Geometric Form of Hyperbolic Hahn Banach Theorem} 
 Hyperbolic hyperplanes have   been defined and discussed thoroughly  in  \cite [Section 8]{Hahn}. In this section, using hyperbolic hyperplanes, we present the  Hahn Banach theorem for hyperbolic modules  in  geometric form.
\begin{definition}
A subset $L$ of a $\mathbb{D}$-module $X$ is said to be a $\mathbb{D}$-linear variety if there exist a $\mathbb{D}$-submodule  $M$ of $X$ and $x_0 \in X$ such that $$L=x_0 + M.$$ 
\end{definition} 
That is, a $\mathbb{D}$-linear variety is a  translate of a $\mathbb{D}$-submodule.
\begin{definition}\cite[Definition 12]{Hahn}
A subset $L$ of a $\mathbb{D}$-module $X$ is said to be a $\mathbb{D}$-hyperplane if there exist a maximal $\mathbb{D}$-submodule  $M$ of $X$ and $x_0 \in X$ such that $$L=x_0 + M.$$ 
\end{definition}
 That is, a $\mathbb{D}$-hyperplane is a  translate of a maximal $\mathbb{D}$-submodule.
\begin{theorem}\cite[Theorem 13(1 and 2)]{Hahn}\label{Ghbt1}
 Let $X$ be a $\mathbb{D}$-module and $L \subset X$. Then the following statements are equivalent:
\begin{enumerate}
\item[(i)] L is a $\mathbb{D}$-hyperplane.
\item[(ii)] There exist a  $\mathbb{D}$-linear functional $f$ that takes at least one invertible value and  $c \in \mathbb{D}$ such that  $$L = \{x \in X \; : \; f(x)=c \}.$$
\end{enumerate}
\end{theorem}

\begin{theorem}\cite[Theorem 13(3)]{Hahn}\label{Ghbt3}
 Let $X$ be a $\mathbb{D}$-module and $L \subset X$ be a $\mathbb{D}$-hyperplane. If $f$ and $g$ are  $\mathbb{D}$-linear functionals on $X$ such that both take at least one invertible value and  $c, d \in \mathbb{D}$ with  $$L = \{x \in X \; : \; f(x)=c \}= \{x \in X \; : \; g(x)=d \}, $$ then, there exists $\gamma \in \mathbb{D}$ such that  $f = \gamma g$ and $c = \gamma d$.
\end{theorem}
\begin{theorem}\cite[Theorem 15]{Hahn}\label{Ghbt2}
 Let $X$ be a $\mathbb{D}$-module and $L \subset X$. Then  $L$  is a $\mathbb{D}$-hyperplane if and only if  $e_1 L $ and $e_2 L $ are real hyperplanes in $e_1 X$ and $e_2 X$ respectively such that $L = e_1 L \oplus e_2 L$.
\end{theorem}

\begin{theorem}\label{Ghbt4}
 Let $X$ be a $\mathbb{D}$-module and $B \subset X$ be a $\mathbb{D}$-convex, $\mathbb{D}$-absorbing set. If $L \subset X$  is a $\mathbb{D}$-hyperplane such that $e_1 B \cap e_1  L= e_2 B \cap e_2 L= \emptyset$, then there exists a $\mathbb{D}$-linear functional $f$ such that 
 $L = \{x \in X \; : \; f(x)=1 \}$ and for each $x \in X$,   $-q_{B}(-x)\leq' f(x) \leq'  q_{B}(x)$.
\end{theorem}
\begin{proof} Since $L$  is a $\mathbb{D}$-hyperplane, by Theorem \ref{Ghbt1}, there exist a $\mathbb{D}$-linear functional $g$ that takes at least one invertible value and  a hyperbolic number $c =e_1 c_1 + e_2 c_2$ such that  $$L = \{x \in X \; : \; g(x)=c \}.$$ Since $g$ is  $\mathbb{D}$-linear on $X$, there exist $\mathbb{R}$-linear functionals $g_1$ and $g_2$ on $e_1 X$ and $e_2 X$ respectively such that for each $x \in X$, we have $$ g(x)= e_1 g_1(e_1 x) + e_2 g_2 (e_2 x) .$$  Also, by Theorem \ref{Ghbt2}, $L = e_1 L \oplus e_2 L$, where  $e_1 L $ and $e_2 L $ are real hyperplanes in $e_1 X$ and $e_2 X$ respectively. It is clear that $e_1 L = \{e_1 x \in e_1 X \; : \; g_1(e_1 x)=c_1 \}$  and   $e_2 L = \{e_2 x \in e_2 X \; : \; g_2(e_2 x)=c_2 \}$.  We claim that $c$ is invertible. Suppose the contrary. Then either $c_1 =0$ or $c_2=0$. Consider the case when $c_1 =0$. Then, $0 \in e_1 L$. Also, $0 \in B$ as $B$  is $\mathbb{D}$-absorbing which yields that $0 \in e_1 B$. Therefore $ 0 \in e_1 B \cap e_1  L$, which  contradicts   the hypothesis that  $e_1 B \cap e_1  L= \emptyset$. Thus $c_1 \not=0$. Similarly, it can be shown that $c_2 \not=0$. This proves our claim that  $c$ is invertible. Take $f= g/c$. Then, by  Theorem \ref{Ghbt3}, we get   $L = \{x \in X \; : \; f(x)=1 \}.$ For each $x \in X$, write $ f(x)= e_1 f_1(e_1 x) + e_2 f_2 (e_2 x),$ where $g_1$ and $g_2$ are $\mathbb{R}$-linear functionals  on $e_1 X$ and $e_2 X$ respectively. Then, we have $e_1 L = \{e_1 x  \; : \; f_1(e_1 x)= 1 \}$  and   $e_2 L = \{e_2 x  \; : \; f_2(e_2 x)= 1 \}$. Now, $B$ is $\mathbb{D}$-convex implies that $e_1B$ is $\mathbb{R}$-convex set in $\mathbb{R}$-vector space $e_1 X$. We also have  $e_1 B \cap e_1  L= \emptyset$.  So, either $e_1 B \subset  \{e_1 x  \; : \; f_1(e_1 x) <1\}$ or $e_1 B \subset  \{e_1 x  \; : \; f_1(e_1 x) >1\}$. As $f_1(0)=0$ and $0 \in e_1 K$, so $e_1 B \subset  \{e_1 x  \; : \; f_1(e_1 x) <1\}$. In a similar fashion, we obtain $e_2 B \subset  \{e_2 x  \; : \; f_2(e_2 x) <1\}$. Now, let $x \in B$. Then,
\begin{eqnarray*}
f(x) &=&  e_1 f_1(e_1 x) + e_2 f_2 (e_2 x)\\
 &<' &  e_1 + e_2\\
 &=&  1.
\end{eqnarray*}
That is, $B \subset \{x \in X \; : \; f(x)<'1 \}$. Let $x \in X$. Then there exists $\gamma >'0$  such that $x \in \gamma B$. That is, $x/       \gamma \in B$. So, $f(x/ \gamma) <' 1$. Hence $f(x) <'   \gamma$. Taking $\mathbb{D}$-infimum over all $\gamma$ for which $x \in \gamma B$, we obtain $f(x) \leq ' q_B (x)$. Consequently for each $x \in X$, we also have $f(-x) \leq ' q_B (-x)$. Thus,  $ -q_B (-x) \leq'  -f(-x) = f(x) \leq '  q_B (x).$
\end{proof}
\begin{theorem}
 Let $X$ be a $\mathbb{D}$-module and $B \subset X$ be a $\mathbb{D}$-convex, $\mathbb{D}$-absorbing set. If $L \subset X$  is a $\mathbb{D}$-linear variety such that $e_1 B \cap e_1  L = e_2 B \cap e_2 L= \emptyset$, then there exists a $\mathbb{D}$-hyperplane $H =  \{x \in X \; : \; f(x)=1 \}$, where $f$  is a $\mathbb{D}$-linear functional on $X$ such that $L \subset H$, for each $x \in X$, $ f(x) \leq'  q_{B}(x)$ and $B \subset  \{x \in X \; : \; f(x)\leq' 1 \}.$
\end{theorem}
\begin{proof} Let $M$ be a  $\mathbb{D}$-submodule  of $X$ and $x_0 \in X$ such that $L = x_0 + M$. Then neither $e_1 x_0 \in e_1 M$ nor  $e_2 x_0 \in e_2 M$. Because if  $e_1 x_0 \in e_1 M$, then $e_1 L = e_1 M$, so $0 \in e_1 L$. Also, $0 \in e_1 B$, as $e_1 B$ is an absorbing set in $e_1 X$, showing that $e_1 B \cap e_1  L \not= \emptyset$, which is a contradiction. Similarly, we arrive at a  contradiction for  $e_2 x_0 \in e_2 M$. Consider the set $N=\textmd{span}(M \cup \{ x_0\})$. Then $N$ is  $\mathbb{D}$-submodule  of $X$, $L \subset N$ and $M$ is a maximal $\mathbb{D}$-submodule  of $N$, and so $L$ is a $\mathbb{D}$-hyperplane in $N$.  Also, $B \cap N$ is a  $\mathbb{D}$-convex, $\mathbb{D}$-absorbing subset of $N$. Since $N$ is a $\mathbb{D}$-submodule, for each $x \in N$, we have
\begin{eqnarray*}
q_{B \cap N}(x) &=&  \inf_{\mathbb{D}} \{ \alpha > ' 0 \; : \; x \in \alpha (B \cap N)\}\\
 &=&  \inf_{\mathbb{D}} \{ \alpha > ' 0 \; : \; x \in \alpha B \cap N \}\\
 &=&  \inf_{\mathbb{D}} \{ \alpha > ' 0 \; : \; x \in \alpha B \}\\
 &=& q_{B }(x).
\end{eqnarray*}
Clearly, $(e_1 B\cap e_1  N) \cap e_1  L=(e_2 B\cap e_2  N) \cap e_2  L= \emptyset$. By Theorem \ref{Ghbt4},  there exists a $\mathbb{D}$-linear functional $g$ on $N$ such that 
 $L = \{x \in X \; : \; g(x)=1 \}$ and for each $x \in N$,   $ g(x) \leq'  q_{B}(x)$. Finally, by applying \cite [Theorem 5] {Hahn}, we obtain a  $\mathbb{D}$-linear functional $f$ on $X$ such that $ f_{|N}= g $ and  for each $x \in X$,   $ f(x) \leq'  q_{B}(x)$. Now, if $x \in B$, then $q_B (x) \leq' 1$. Thus $f(x)  \leq'  q_{B}(x) \leq' 1$, showing that $B \subset  \{x \in X \; : \; f(x)\leq' 1 \}.$  Let $H =  \{x \in X \; : \; f(x)=1 \}$. Then, $H$   is a $\mathbb{D}$-hyperplane in $X$ and  $L \subset H$. This completes the proof.
\end{proof}

\end{section}


\bibliographystyle{amsplain}

\begin{thebibliography}{99}
%
\bibitem{YY} D. Alpay, M. E. Luna-Elizarraras, M. Shapiro and D. C. Struppa, {\em Basics of Functional Analysis with Bicomplex scalars, and Bicomplex Schur Analysis}, Springer Briefs in Mathematics, 2014.  
%
\bibitem{BA} S. Banach,  {\em Sur les fonctionelles lin\'{e}aires II},  Studia Math. \textbf{1}, (1929),  223-239.
%
\bibitem{BAHS} S. Banach and H. Steinhaus,  {\em Sur le principe de la condensation de $singularit\dot{e}s$},  Studia Math. \textbf{2}, (1927),  50-61.
%
\bibitem{BS} H. F. Bohnenblust and A. Sobczyk, {\em Extensions of functionals on complex linear spaces},  Bull.  Amer. Math. Soc. \textbf{44}, No. 2 (1938), 91-93.
%
 \bibitem{CS}  F. Colombo, I. Sabadini and D. C. Struppa, {\em Bicomplex holomorphic functional calculus}, Math. Nachr. \textbf{287}, No. 13 (2013), 1093-1105. 
%
\bibitem{H_6H_6}  F. Colombo, I. Sabadini, D. C. Struppa,  A. Vajiac and M. B. Vajiac, {\em Singularities of functions of one and several bicomplex variables}, Ark. Mat. \textbf{49}, (2011), 277-294.
%
\bibitem{JJ} J. B. Conway, {\em A Course in Functional Analysis}, 2nd Edition, Springer, Berlin, 1990.
%
\bibitem{ff} H. De Bie, D. C. Struppa, A. Vajiac and M. B. Vajiac, {\em The Cauchy-Kowalewski product for bicomplex holomorphic functions}, Math. Nachr. \textbf{285}, No. 10 (2012), 1230-1242.
%
\bibitem{dusc}
N. Dunford and J. T. Schwartz { \em Linear Operators, Part 1: General Theory}, Interscience Publishers, New York-London, 1958.
%
 \bibitem{G_1G_1} R. Gervais Lavoie, L. Marchildon and D. Rochon, {\em Finite-dimensional bicomplex Hilbert spaces}, Adv. Appl. Clifford Algebr. \textbf{21}, No. 3 (2011), 561-581.
%
\bibitem{GG} R. Gervais Lavoie, L. Marchildon and D. Rochon, {\em Infinite-dimensional bicomplex Hilbert spaces}, Ann. Funct. Anal. \textbf{1}, No. 2 (2010), 75-91.
%
\bibitem{HH0} H. Hahn, {\em $\ddot{U}ber$ die Darstellung gegebener Funktionen durch singulare Integr$\bar{a}$le II}, Denkschriften der K. Akad. Wien. Math. Naturwiss. Kl. \textbf{93}, (1916),  657-692.
%
\bibitem{HH} H. Hahn, {\em $\ddot{U}ber$ lineare Gleichungsysteme in linearen R\"{a}ume} , J. Reine Angew.
Math.  \textbf{157}, (1927),  214-229.
%
\bibitem{LL} R. Kumar, R. Kumar and D. Rochon, {\em The fundamental theorems in the framework of bicomplex topological modules}, (2011), arXiv:1109.3424v1.
%
\bibitem{HS} R. Kumar  and  H. Saini,  {\em Topological Bicomplex Modules},  Adv. Appl. Clifford Algebr., \textbf{26}, No. 4 (2016), 1249-1270.
%
\bibitem{RK} R. Kumar, K. Singh, H. Saini and S. Kumar, {\em Bicomplex weighted Hardy  spaces and bicomplex C$^{*}$-algebras},  Adv. Appl. Clifford Algebr. \textbf{26}, No. 1 (2016), 217-235.
%
\bibitem{LA} R. Larsen, {\em Functional Analysis : An Introduction}, Marcel Dekker, New York, 1973.
%
\bibitem{HL} H. Lebesgue,  {\em Sur les $int\acute{e}grales$  $singuli\grave{e}res$},  Ann. de Toulouse \textbf{1}, (1909),  25-117.
%
\bibitem{Hahn} M. E. Luna-Elizarraras, C. O. Perez-Regalado and M. Shapiro, {\em On linear functionals and Hahn-Banach theorems for hyperbolic and bicomplex modules}, Adv. Appl. Clifford Algebr. \textbf{24}, (2014), 1105-1129.    
%
\bibitem{Luna} M. E. Luna-Elizarraras, C. O. Perez-Regalado and M. Shapiro, {\em On the bicomplex Gleason-Kahane Zelazko Theorem}, Complex Anal. Oper. Theory, \textbf{10}, No. 2 (2016), 327-352.
%
\bibitem{MM} M. E. Luna-Elizarraras, M. Shapiro and D. C. Struppa, {\em On Clifford analysis for holomorphic mappings}, Adv. Geom. \textbf{14}, No. 3 (2014), 413-426.
%
\bibitem{M_1M_1} M. E. Luna-Elizarraras, M. Shapiro, D. C. Struppa and A. Vajiac, {\em Bicomplex numbers and their elementary functions}, Cubo \textbf{14}, No. 2 (2012), 61-80.

%
\bibitem{ZZE} M. E. Luna-Elizarraras, M. Shapiro, D. C. Struppa and A. Vajiac, {\em Bicomplex Holomorphic Functions : The Algebra, Geometry and Analysis of Bicomplex Numbers}, Frontiers in Mathematics, Springer, New York, 2015.
%
\bibitem{ZZ} M. E. Luna-Elizarraras, M. Shapiro, D. C. Struppa and A. Vajiac, {\em Complex Laplacian and derivatives of bicomplex functions}, Complex Anal. Oper. Theory \textbf{7} (2013), 1675-1711.  
%
\bibitem{KK} G. B. Price, {\em An Introduction to Multicomplex Spaces and Functions}, 3rd Edition, Marcel Dekker, New York, 1991.
%
\bibitem{Z_1Z_1} J. D. Riley, {\em Contributions to the theory of functions of a bicomplex variable}, Tohoku Math J(2) \textbf{5}, No.2 (1953), 132-165.
%
\bibitem{RR} D. Rochon and M. Shapiro, {\em On algebraic properties of bicomplex and hyperbolic numbers}, Anal. Univ. Oradea, Fasc. Math. \textbf{11} (2004), 71-110.
%
\bibitem{XX} D. Rochon and S. Tremblay, {\em Bicomplex Quantum Mechanics II: The Hilbert Space},Adv. Appl. Clifford Algebr.  \textbf{16} No. 2 (2006), 135-157.
 %
\bibitem{rudin} W. Rudin, {\em Functional analysis}, 2nd Edition, McGraw Hill, New York, 1991.
%
 %
\bibitem{sktm} S. Saks and J. D. Tamarkin, {\em On a theorem of Hahn-Steinhaus},  Ann. of Math. \textbf{2} No. 34 (1933), 595-601.
%
 %
\bibitem{schau} J. Schauder, {\em $\ddot{U}ber$ die Umkehrung linearer, stetiger funktionaloperationen}, Studia Math.  \textbf{2}, (1930), 1-6. 
%
\bibitem{souk} G. Soukhomlinoff, {\em $\ddot{U}ber$ Fortsetzung von linearen Funktionalen in linearen komplexen $R\ddot{a}umen$ und linearen Quaternionr$\ddot{a}$umen}, Rec. Math. N.S. \textbf{3}  No. 2 (1938), 353-358.

%
\bibitem{stein} H. Steinhaus, {\em  Sur les $d\acute{e}veloppements$ orthogonaux}, Bull. Int. Acad. Polon. Sci. A. (1926), 11-39.
\end{thebibliography}

\noindent Heera Saini, \textit{Department of Mathematics, Govt. M. A. M. College, Jammu,  J\&K - 180 006, India.}\\
E-mail :\textit{ heerasainihs@gmail.com}\\

\noindent Aditi Sharma, \textit{Department of Mathematics, University of Jammu, Jammu, J\&K - 180 006, India.}\\
E-mail :\textit{aditi.sharmaro@gmail.com}\\

\noindent Romesh Kumar, \textit{Department of Mathematics, University of Jammu, Jammu, J\&K - 180 006, India.}\\
E-mail :\textit{ romesh\_jammu@yahoo.com}\\

\end{document}